\documentclass{article}

\usepackage{amsmath, amssymb,epic,graphicx,mathrsfs,enumerate}

\usepackage{amsfonts}
\usepackage{amsthm}
\usepackage{latexsym}
\usepackage{longtable}
\usepackage{epsfig}
\usepackage{amsmath}

\newtheorem{teor}{Theorem}
\newtheorem{lemma}[teor]{Lemma}

\begin{document}

\title{Irredundant and minimal covers of finite groups}
\date{}
\author{Andrea Lucchini \\Department of Mathematics\\University of Padova (Italy) \and Martino Garonzi \\Department of Mathematics\\University of Padova (Italy)}
\maketitle
\begin{abstract}
A cover of a finite non-cyclic group $G$ is a family $\mathcal{H}$ of proper subgroups of $G$ whose union equals $G$. A cover of $G$ is called minimal if it has minimal size, and irredundant if it does not properly contain any other cover. We classify the finite non-cyclic groups all of whose irredundant covers are minimal.
\end{abstract}

\section{Introduction}

Let $G$ be a finite non-cyclic group. A cover of $G$ is a family $\mathcal{H}$ of proper subgroups of $G$ such that $G = \bigcup_{H \in \mathcal{H}} H$. Covers of groups have been studied by many authors. There are at least two interesting notions of \textit{minimality} for a cover: size-wise or inclusion-wise. Let $\mathcal{H}$ be a cover of $G$. $\mathcal{H}$ is said to be minimal (size-wise minimal) if it has minimum size among all covers of $G$. $\mathcal{H}$ is said to be irredundant (or inclusion-wise minimal) if no proper subfamily of $\mathcal{H}$ is a cover of $G$. Obviously, a minimal cover is always irredundant but the converse is not always true: for instance if $n > 2$ and $p$ is a prime, the non-trivial cyclic subgroups of the elementary abelian group ${C_p}^n$ form an irredundant cover (they pairwise intersect in the identity) but only $p+1$ proper subgroups are needed to cover the group: just lift a cover of any epimorphic image of the form $C_p \times C_p$.

It is natural to ask when these two minimality notions collapse, that is, to ask when are all irredundant covers size-wise minimal. In other words, we consider the finite non-cyclic groups all of whose irredundant covers have the same size, in short we will say that they ``admit only one-sized covers''. Such groups were considered by J. R. Rog\'erio in \cite{rogerio}. Note that if $N \unlhd G$ and $\mathcal{H}$ is an irredundant cover of $G/N$ then the family of preimages of the members of $\mathcal{H}$ via the projection $G \to G/N$ is an irredundant cover of $G$, therefore if $G$ admits only one-sized covers and $N \unlhd G$ with $G/N$ non-cyclic then $G/N$ admits only one-sized covers. J. R. Rog\'erio in \cite{rogerio} proved that a nilpotent group which admits only one-sized covers is necessarily of the form $H \times C$ where $|H|$ and $|C|$ are coprime, $C$ is cyclic and $H$ is either the quaternion group $Q_8$ or $C_p \times C_p$ for $p$ a prime, and that a solvable group which admits only one-sized covers is supersolvable. In this paper we give a complete classification of groups with only one-sized covers. The only strong tools we use for this classification are illustrated in Section \ref{tools}.

The following is our main result.

\begin{teor} \label{main}
Let $G$ be a finite non-cyclic group. Then $G$ admits only one-sized covers if and only if $G = H \times C$ where $C$ is a cyclic group of order coprime to $|H|$ and one of the following facts happens.
\begin{enumerate}
\item $H \cong C_p \times C_p$ for some prime $p$,
\item $H \cong Q_8$,
\item $H \cong C_p \rtimes C_n$ where $p$ is a prime that does not divide $n$ and the action of $C_n$ on $C_p$ is non-trivial.
\end{enumerate}
\end{teor}

Let $G$ admit only one-sized covers, and write $G = H \times C$ as in the theorem. Any irredundant cover of $G$ has the form $\{K \times C\ :\ K \in \mathcal{H}\}$ where $\mathcal{H}$ is an irredundant cover of $H$. As Rog\'erio observed, in a nilpotent group with only one-sized covers the family of maximal cyclic subgroups is the unique irredundant cover. In the non-nilpotent case $H \cong C_p \rtimes C_n$ (case (3) of the theorem) every irredundant cover of $H$ contains all the $p$-complements (there are $p$ of them, they are cyclic of order $n$) and any proper subgroup containing the centralizer of the Sylow $p$-subgroup.

\section{The main tools} \label{tools}

Let us illustrate the known results concerning covers that we will need in our proof. Let $\sigma(G)$ be the minimal size of a cover of the finite group $G$, where we put $\sigma(G) = \infty$ if $G$ is cyclic. Note that if $N \unlhd G$ then $\sigma(G) \leq \sigma(G/N)$, since a cover of $G/N$ can always be lifted to a cover of $G$. The number $\sigma(G)$ was computed by Tomkinson in \cite{tom} for solvable groups. He proved the following.

\begin{teor} \label{tom}
Let $G$ be a solvable non-cyclic group. Then $\sigma(G)$ equals $q+1$ where $q$ is the smallest order of a chief factor of $G$ with multiple complements.
\end{teor}

The following is a corollary of the main result of Bryce and Serena in \cite{serena}.

\begin{teor} \label{serena}
Let $G$ be a non-cyclic group. If $G$ admits a cover of size $\sigma(G)$ consisting of abelian subgroups then $G$ is solvable.
\end{teor}

\section{Proof of Theorem \ref{main}}

We start with some lemmas. Recall that a ``maximal cyclic subgroup'' of $G$ is a cyclic subgroup of $G$ which is not properly contained in any other cyclic subgroup of $G$. Define $\lambda(G)$ to be the maximum size of an irredundant cover of $G$.

\begin{lemma} \label{osclemma}
Let $\mathcal{C}$ be the family of all maximal cyclic subgroups of a finite non-cyclic group $G$. Then
\begin{enumerate}
\item $\mathcal{C}$ is an irredundant cover of $G$.
\item For any irredundant cover $\mathcal{H}$ of $G$, $|\mathcal{H}| \leq |\mathcal{C}|$.
\item Let $\mathcal{H}$ be an irredundant cover of $G$. Then $|\mathcal{H}| = |\mathcal{C}|$ if and only if each $H \in \mathcal{H}$ contains exactly one maximal cyclic subgroup of $G$.
\end{enumerate}
In particular $\lambda(G)$ equals the number of maximal cyclic subgroups of $G$.

\

Now assume that $G$ admits only one-sized covers. Then the following hold.
\begin{enumerate}
\setcounter{enumi}{3}
\item $\lambda(G) = \sigma(G)$.
\item If $N \unlhd G$ and $G/N$ is non-cyclic then $\sigma(G) = \sigma(G/N)$ and $G/N$ has precisely $\sigma(G)$ maximal cyclic subgroups.
\item If $C_1,C_2$ are two distinct maximal cyclic subgroups of $G$, $\langle C_1 \cup C_2 \rangle = G$.
\end{enumerate}
\end{lemma}

\begin{proof}
(1) is clear. To prove (2) let $\mathcal{H}$ be an irredundant cover of $G$, call $k := |\mathcal{C}|$ and let $x_1,\ldots,x_k$ be fixed generators of the distinct maximal cyclic subgroups of $G$. Since $\mathcal{H}$ is a cover of $G$, for $i = 1,\ldots,k$ we can choose $H_i \in \mathcal{H}$ such that $x_i \in H_i$ (where it can happen that $H_i = H_j$ for $i \neq j$). Since any element of $G$ is contained in a maximal cyclic subgroup of $G$, any element of $G$ is a power of some $x_i$, hence $\{H_1,\ldots,H_k\}$ is a cover of $G$ contained in $\mathcal{H}$. Since $\mathcal{H}$ is an irredundant cover $\mathcal{H} = \{H_1,\ldots,H_k\}$ hence $|\mathcal{H}| \leq k = |\mathcal{C}|$. This proves (2). If equality holds, i.e. $|\mathcal{H}| = k$, then the $H_i$'s are pairwise distinct hence each $H_i$ contains exactly one maximal cyclic subgroup of $G$. This proves (3).
(4) follows from the fact that covers of minimal size are clearly irredundant, hence they have size $\lambda(G)$ as $G$ admits only one-sized covers.
We prove (5). Since $G/N$ is non-cyclic, the cover $\mathcal{C}$ of $G$ consisting of the maximal cyclic subgroups maps via the projection $G \to G/N$ to a cover of $G/N$ of the same size consisting of cyclic subgroups. Hence if $G$ admits only one-sized covers then $\sigma(G) = |\mathcal{C}| \geq \sigma(G/N)$. Since $\sigma(G) \leq \sigma(G/N)$ we get equality: $\sigma(G) = \sigma(G/N)$. Since $G$ admits only one-sized covers, so does $G/N$ hence by (1) it has $\sigma(G/N) = \sigma(G)$ maximal cyclic subgroups.
We prove (6). Let $C_1,C_2,\ldots,C_k$ be the maximal cyclic subgroups of $G$. If $\langle C_1 \cup C_2 \rangle \neq G$ then $\{\langle C_1 \cup C_2 \rangle,C_3,\ldots,C_k\}$ would be a cover of $G$ of size $k-1 = \lambda(G)-1$, contradicting (4).
\end{proof}

\begin{lemma} \label{pnilp}
Let $G$ be a finite solvable group and let $p$ be a prime divisor of $|G|$. Assume that all complemented chief factors of $G$ which are $p$-groups are central. Then $G$ is $p$-nilpotent.
\end{lemma}

\begin{proof}
By induction on $|G|$. Let $N$ be a minimal normal subgroup of $G$. Then $G/N$ is $p$-nilpotent by induction. Let $\Phi(G)$ denote the Frattini subgroup of $G$. If $N \subseteq \Phi(G)$ then $G$ is $p$-nilpotent by \cite{doerkhawkes}, Lemma 13.2 in Chapter A, so now assume $N$ is non-Frattini, i.e. $N$ is complemented. If $N$ is a $p$-group then by hypothesis $N$ is central hence $G = H \times N$ with $H$ $p$-nilpotent, hence $G$ is $p$-nilpotent. Now assume $N$ is not a $p$-group. Let $K/N$ be a normal $p$-complement of $G/N$. Since $p$ does not divide $|N|$, $p$ does not divide $|K|$ hence $K$ is a normal $p$-complement of $G$ thus $G$ is $p$-nilpotent.
\end{proof}

\begin{lemma} \label{cycfrob}
Let $G$ be a group with a maximal core-free cyclic subgroup $H$ and a normal subgroup $N$ complemented by $H$. Then $N$ with the conjugates of $H$ form an irredundant cover of $G$ of size $|N|+1$.
\end{lemma}

\begin{proof}
If $K$ is a conjugate of $H$ different from $H$ then $H \cap K$ is normal in $H$ and $K$ because $H$ and $K$ are cyclic, and $\langle H,K \rangle = G$ because $H,K$ are maximal subgroups of $G$. Therefore $H \cap K \unlhd G$. Since $H$ is core-free, $H \cap K = \{1\}$. It follows that $N$ with the conjugates of $H$ cover exactly $$|N| + (|H|-1)|N| = |N||H| = |G|$$ elements, hence they form a cover $\mathcal{H}$ of $G$ which is irredundant because any two members of $\mathcal{H}$ intersect in $\{1\}$.
\end{proof}

Now we proceed with the proof of Theorem \ref{main}.

\

Let $G$ admit only one-sized covers. Let $\mathcal{C}$ be the family of maximal cyclic subgroups of $G$. From Lemma \ref{osclemma} it follows that $|\mathcal{C}| = \sigma(G)$. It follows by Theorem \ref{serena} that $G$ is solvable. We prove that $G$ is supersolvable. By \cite{rob} 5.4.7 and 9.4.4 a finite group is supersolvable if and only if all of its maximal subgroups have prime index. Let $M$ be a maximal subgroup of $G$, and let $|G:M| = q = p^k$ a prime power. We need to show that $q$ is a prime, i.e. $k=1$. Let $X := G/M_G$ where $M_G$ is the normal core of $M$ in $G$. $X$ is a primitive solvable group of degree $q$, hence $X = V \rtimes H$ with $|V|=q$ and $H$ an irreducible subgroup of $GL(V)$. If $H = \{1\}$ then $q = |V|$ is a prime. Now assume $H \neq \{1\}$, so that $X$ is non-nilpotent. Since $G$ admits only one-sized covers, so does $X$ by Lemma \ref{osclemma} (5). Let $\ell$ be a prime dividing $|H|$ but not $|V|$. Let $h \in H$ of order $\ell$ and let $\langle x \rangle$ a maximal cyclic subgroup of $X$ containing $h$. Note that if $v \in V$ normalizes $\langle h \rangle$ then $[v,h] = vhv^{-1}h^{-1} \in H \cap V = \{1\}$ hence $v$ centralizes $h$. Since $C_H(V) = \{1\}$ it follows that there exists $v \in V$ with $\langle h \rangle^v \neq \langle h \rangle$. In particular $\langle x \rangle \neq \langle x^v \rangle$ hence $\langle x \rangle$ and $\langle x^v \rangle$ are two distinct maximal cyclic subgroups of $X$. Therefore by Lemma \ref{osclemma} (6) since $X$ admits only one-sized covers $X = \langle x,x^v \rangle \subseteq V \langle x \rangle \subseteq X$ hence $X = V \langle x \rangle$ and hence $H$ is cyclic (actually $H = \langle x \rangle$). By Lemma \ref{cycfrob} it follows that $X$ has an irredundant cover of size $q+1$, hence $\lambda(X) = \sigma(X) = q+1$. If $V$ was not cyclic then writing $V$ as union of cyclic subgroups intersecting in the identity would give, with the $q$ conjugates of $H$, an irredundant cover of $X$ of size larger than $q+1$. Hence $V$ is cyclic, i.e. $q = p$.

We deal with the nilpotent case and with the supersolvable non-nilpotent case separately.

First, let $G$ be nilpotent. Let $P$ be a Sylow $p$-subgroup of $G$. Then $P$ is an epimorphic image of $G$ hence if $P$ is non-cyclic, by Lemma \ref{osclemma} (5) $\sigma(G) = \sigma(P) = p+1$. This means that exactly one of the Sylow subgroups of $G$ is non-cyclic, hence $G$ has the form $P \times H$ with $H$ a cyclic Hall direct factor and $P$ a Sylow $p$-subgroup. Hence we may assume that $G = P$, i.e. $G$ is a $p$-group. We show that all maximal subgroups of $G$ are cyclic. Let $M_1,\ldots,M_k$ be the maximal subgroups of $G$. By Lemma \ref{osclemma}(6) $G$ is $2$-generated hence $G/\Phi(G) \cong C_p \times C_p$, so that $k = p+1$. Since $G$ has exactly $p+1$ maximal cyclic subgroups, we are left to show that all maximal cyclic subgroups of $G$ are maximal subgroups. By Lemma \ref{osclemma} (3) each $M_i$ contains exactly one maximal cyclic subgroup of $G$. Let $C$ be a maximal cyclic subgroup of $G$. Without loss of generality $C \subseteq M_1$. Then $C \cup M_2 \cup \ldots \cup M_k = G$ hence setting $c := |C|$ since $|\Phi(G)| = p^{n-2}$ we must have $$p(p^{n-1}-p^{n-2})+p^{n-2} + c > p^n,$$and this implies $c > p^{n-2}(p-1)$. Since $c$ is a power of $p$ we deduce $c = p^{n-1}$, i.e. $C$ is a maximal subgroup of $G$. This proves that all maximal subgroups of $G$ are cyclic. Let now $x$ be a central element of $G$ of order $p$. If $G$ has an element $y$ of order $p$ which is not a power of $x$ then $G \geq \langle x \rangle \langle y \rangle \cong C_p \times C_p$. If $\langle x \rangle \langle y \rangle$ equals $G$ then $G \cong C_p \times C_p$. Suppose this is not the case. Then $G$ has a proper non-cyclic subgroup $\langle x \rangle \langle y \rangle$, and this contradicts the fact that all maximal subgroups of $G$ are cyclic. Thus $\langle x \rangle$ is the unique subgroup of $G$ of order $p$, hence by \cite{rob} 5.3.6 $G$ is a generalized quaternion group. Let $N$ be the unique subgroup of $G$ of order $2$. Then $G/N$ is both a dihedral $2$-group and a $2$-group which admits only one-sized covers, hence by what we have proved either $G \cong C_2 \times C_2$ or $G/N$ is a generalized quaternion group. Therefore $G/N \cong C_2 \times C_2$, i.e. $G \cong Q_8$.

Let now $G$ be supersolvable and non-nilpotent. Since $G$ is non-nilpotent there exists a non-central complemented chief factor $C$ of $G$, with $|C|=p$. Write $C = H/K$ and let $A/K$ be a complement of $H/K$ in $G/K$. Let $X := (G/K)/(C_{A/K}(H/K))$. Then $X$ is a primitive group and $$(H/K)(C_{A/K}(H/K))/(C_{A/K}(H/K)) \cong C$$ is its unique minimal normal subgroup (which we will call again $C$), which is complemented by $(A/K)/(C_{A/K}(H/K))$. Since $C$ has prime order, the complements of $C$ in $X$ are cyclic. By Lemma \ref{osclemma}(5) and Lemma \ref{cycfrob} it follows that $\sigma(G) = \sigma(X) = p+1$. Therefore all non-central complemented chief factors of $G$ have the same order $p$. It follows that if a complemented chief factor of $G$ has order $q \neq p$ then it is central. By Lemma \ref{pnilp} it follows that $G$ is $q$-nilpotent for all prime divisors $q$ of $|G|$ different from $p$. The intersection of the normal $q$-complements for $q \neq p$ is then equal to the unique Sylow $p$-subgroup of $G$, call it $P$. By the Schur-Zassenhaus theorem there is a complement $H$ of $P$, hence $G = P \rtimes H$.

If $H$ was not cyclic then $p+1 = \sigma(G) = \sigma(H)$, contradicting Theorem \ref{tom} because $H$ is solvable and $|H|$ is coprime to $p$. Hence $H = \langle a \rangle$ is cyclic. Since $G$ is non-nilpotent $H$ does not centralize $P$. Let $F := \Phi(P)$. By \cite{rob} 5.3.3 $H$ does not centralize $P/F$. Hence $G/F$ admits a non-cyclic quotient of the form $C_p \rtimes C_n$, being $n$ a positive integer, and by Lemma \ref{osclemma} (5) and Lemma \ref{cycfrob} $\sigma(G) = \sigma(G/F) = p+1$. Write $P/F = C_p^t$. The cyclic subgroups of $G/F$ have order at most $pn$, on the other hand $G/F$ is the union of $p+1$ cyclic subgroups hence $(p+1)(pn-1)+1 \geq p^t n$ and we deduce $t \leq 2$. Suppose $t=2$. $G/F$ must have central elements of order $p$ otherwise the cyclic subgroups would have order $n$ or $p$ and they couldn't cover $G/F$; but then $G/F \cong C_p \times H$ with $H = C_p \rtimes C_n$ non-abelian: this group has more than $p+1$ maximal cyclic subgroups. Therefore $t=1$, i.e. $P$ is cyclic, say $P \cong C_{p^k}$.

We are left to prove that $k=1$. Up to quotienting out by $C_H(P)$ we may assume that $C_H(P) = \{1\}$ hence $P$ is a maximal cyclic subgroup of $G$. Let $A := C_P(H)$. Then $AH$ is a maximal cyclic subgroup of $G$ and it has at least $p$ conjugates because $AH$ is not normal in $G$ being $H$ not normal in $G$. Also, every conjugate of $AH$ contains $A$. Since $G$ has exactly $p+1$ maximal cyclic subgroups which form a cover we must then have $$p(|AH|-|A|)+|P| \geq |G| = |P||H|,$$and since $A \neq P$ we deduce $|P:A|=p$. We must show that $A = \{1\}$ (this is equivalent to $k=1$). $H = \langle a \rangle$ acts on $P = \langle x \rangle$ by raising $x$ to a power $l$. Since $H$ centralizes $A = \langle x^p \rangle$ we have $x^p = a^{-1}x^pa = x^{pl}$ hence $pl \equiv p \mod p^k$, i.e. $p^{k-1}$ divides $l-1$, in particular $l$ is coprime to $p$. On the other hand since $H$ acts faithfully (because $C_H(P) = \{1\}$) $l$ divides $\varphi(p^k) = p^{k-1}(p-1)$, so since $(l,p)=1$, $l$ divides $p-1$. But then $p^{k-1} \leq l \leq p-1$ which forces $k=1$.


\begin{thebibliography}{10}
\bibitem{serena} R. A. Bryce, L. Serena, A note on minimal coverings of groups by subgroups. Special issue on group theory. J. Aust. Math. Soc. 71 (2001), no. 2, 159-–168.
\bibitem{tom} M. J. Tomkinson, Groups as the union of proper subgroups; Math. Scand., 81(2) (1997), 191--198.
\bibitem{rob} D. J. S. Robinson, A course in the theory of groups; Second edition; Graduate Texts in Mathematics, 80. Springer-Verlag, New York, 1996.
\bibitem{doerkhawkes} K. Doerk, T. Hawkes, Finite soluble groups. De Gruyter Expositions in Mathematics, 4. Walter de Gruyter \& Co., Berlin, 1992.
\bibitem{rogerio} J. R. Rog\'erio, A note on maximal coverings of groups. Comm. Algebra 42 (2014), no. 10, 4498–-4508.
\end{thebibliography}
\end{document}